\documentclass[preprint]{elsarticle}

\journal{}

\usepackage{graphicx}
\usepackage{amssymb,amsmath,amsbsy}
\usepackage{color}
\usepackage{hyperref}
\usepackage{float}

\newfloat{Listing}{tbhp}{lst}

\newcommand{\ul}{\underline}
\newcommand{\dd}{\textnormal{d}}

\newtheorem{theorem}{Theorem}
\newtheorem{lemma}[theorem]{Lemma}
\newdefinition{remark}{Remark}
\newproof{proof}{Proof}
\newdefinition{definition}{Definition}

\newcommand{\qedhere}{\tag*{\qed}}

\makeatletter
\def\ps@pprintTitle{%
\let\@oddhead\@empty
\let\@evenhead\@empty
\def\@oddfoot{\footnotesize\itshape
 \hfill\today}%
\let\@evenfoot\@oddfoot}
\makeatother

\begin{document}

\begin{frontmatter}

\title{A Robust Multigrid Method for Isogeometric Analysis using Boundary Correction}

\author[numa]{Clemens Hofreither}
\ead{chofreither@numa.uni-linz.ac.at}
\author[ricam]{Stefan Takacs}
\ead{stefan.takacs@ricam.oeaw.ac.at}
\author[numa]{Walter Zulehner}
\ead{zulehner@numa.uni-linz.ac.at}

\address[numa]{Department of Computational Mathematics, Johannes Kepler University Linz, Austria}

\address[ricam]{Johann Radon Institute for Computational and Applied Mathematics (RICAM),\\
Austrian Academy of Sciences}

\date{\today}

\begin{abstract}
    We consider geometric multigrid methods for
	the solution of linear systems arising from isogeometric discretizations
	of elliptic partial differential equations.
	For classical finite elements, such methods are well known to be fast solvers
    showing optimal convergence behavior.
	However, the naive application of multigrid to the isogeometric case results in
    significant deterioration of the convergence rates if the spline degree is increased.

    Recently, a robust approximation error estimate and a corresponding inverse inequality for
	B-splines of maximum smoothness have been shown, both with constants independent of
	the spline degree.
    We use these results to construct multigrid solvers for discretizations based on
    B-splines with maximum smoothness which exhibit robust convergence rates.
\end{abstract}


\end{frontmatter}


\section{Introduction}\label{sec:1}

Isogeometric Analysis (IgA), introduced by Hughes et al.~\cite{Hughes:2005},
is an approach to the discretization of
partial differential equations (PDEs) which aims to bring geometric modeling
and numerical simulation closer together.
The fundamental idea is to use spaces of B-splines or non-uniform rational B-splines (NURBS)
both for the geometric description of the computational domain
and as discretization spaces for the numerical solution of PDEs on such domains.
As for classical finite element methods (FEM), this leads to
linear systems with large, sparse matrices.
A good approximation of the solution of the PDE requires sufficient refinement, which
causes both the dimension and the condition number of the stiffness matrix to grow.
At least for problems on three-dimensional domains or
the space-time cylinder, the use of direct solvers does not seem feasible.
The development of efficient linear solvers or preconditioners for such
linear systems is therefore essential.

For classical FEM, it is well-known that hierarchical methods, like
multigrid and multilevel methods, are very efficient and show optimal complexity, that is,
the required number of iterations for reaching a fixed accuracy goal is independent of the grid
size. In this case, the overall computational complexity of the method grows only linearly
with the number of unknowns.

It therefore seems natural to extend these methods to IgA, and several
results in this direction can be found in the literature.
Multigrid methods for IgA based on classical concepts have been considered in
\cite{GahalautEtAl:2013,HofreitherZulehner:2014,HofreitherZulehner:2014c},
and a classical multilevel method in \cite{BuffaEtAl:2013}.
It has been shown early on that a standard approach to constructing
geometric multigrid solvers for IgA
leads to methods which are robust in the grid size~\cite{GahalautEtAl:2013}.
However, it has been observed that the resulting
convergence rates deteriorate significantly when the spline degree
is increased. Even for moderate choices like splines of degree four,
too many iterations are required for practical purposes.

Recently, in~\cite{HofreitherZulehner:2014b}
some progress has been made by using a Richardson method
preconditioned with the mass matrix as a smoother (\emph{mass-Richardson smoother}).
The idea is to carry over the concept of operator
preconditioning to multigrid smoothing: here the (inverse of) the mass matrix can
be understood as a Riesz isomorphism representing the standard $L^2$-norm, the
Hilbert space where the classical multigrid convergence analysis is developed.
Local Fourier analysis
indicates that a multigrid method equipped with such a smoother should show convergence rates that are
independent of both the grid size and the spline degree.
However, numerical results indicate that the proposed method is not robust
in the spline degree in practice. This is due to boundary effects, which cannot be
captured by local Fourier analysis.
Similar techniques have been used to construct symbol-based multigrid
approaches for IgA, see~\cite{Donatelli:2014a,Donatelli:2014b,Donatelli:2015}.

In the present paper, we take a closer look at the origin of these boundary effects
and introduce a boundary correction that deals with these effects. We prove that
the multigrid method equipped with a properly corrected mass-Richardson smoother converges
robustly both in the grid size and the spline degree
for one- and two-dimensional problems.
For the proof, we make use of the results of the recent paper \cite{Takacs:Takacs:2015},
where robust approximation error estimates and robust inverse
estimates for splines of maximum smoothness have been shown.
We present numerical results that illustrate the theoretical results.
Throughout the paper we restrict ourselves to the case of splines with maximum continuity,
which is of particular interest in IgA.

The bulk of our analysis is first carried out in the one-dimensional setting.
While multigrid solvers are typically not interesting in this setting from
a practical point of view, the tensor product structure of the spline spaces
commonly used in IgA lends itself very well to first analyzing the
one-dimensional case and then extending the results into higher dimensions.
In the present work, we extend the ideas from the one-dimensional to the
two-dimensional case and obtain a robust and efficiently realizable smoother
for that setting.

The paper is organized as follows.
In Section~\ref{sec:2}, we introduce the spline spaces used for the discretization and the elliptic model problem.
A general multigrid framework is introduced in Section~\ref{sec:mg}.
This framework requires a proper choice of the smoother,
which is discussed in detail for one-dimensional domains
in Section~\ref{sec:1d}. In Section~\ref{sec:2d}, we
extend these results to the case of two-dimensional domains.
Some details on the numerical realization of the proposed smoother as well as
numerical experiments illustrating the theory are given in Section~\ref{sec:num}.
In Section~\ref{sec:conc}, we close with some concluding remarks.
A few proofs are postponed to the Appendix.


\section{Preliminaries}\label{sec:2}

\subsection{B-splines and tensor product B-splines}\label{subsec:2a}

First, we consider the case of one-dimensional domains and assume,
without loss of generality, $\Omega = (0,1)$.
We introduce for any $\ell\in \mathbb{N}_0=\{0,1,2,\ldots\}$ a uniform subdivision (grid)
by splitting $\Omega$ into $n_\ell=n_0 2^\ell$ subintervals of length $h_\ell:=\tfrac1{n_\ell} = \tfrac1{n_0}2^{-\ell}$
by uniform dyadic refinement.
On these grids we introduce spaces of spline functions with maximum smoothness as follows.

\begin{definition}
  \label{def:1}
        For $p\in \mathbb{N}:=\{1,2,3,\ldots\}$,
	$S_{p,\ell}(0,1)$ is the space of all functions in $C^{p-1}(0,1)$ which are
	polynomials of degree $p$ on each subinterval $((i-1) h_\ell, i h_\ell)$ for $i=1,\ldots,n_\ell$.
        Here $C^m(\Omega)$ denotes the space of all continuous functions mapping
        $\Omega\rightarrow \mathbb{R}$ that are $m$ times continuously differentiable.
\end{definition}

We assume that the coarsest grid is chosen such that $n_0>p$ holds.
Note that the spaces are nested for fixed spline degree $p$,
that is, $S_{p,\ell-1}(0,1) \subset S_{p,\ell}(0,1)$,
and the number of degrees of freedom roughly doubles in each refinement step.
The parameter $\ell$ will play the role of the grid level in the construction
of our multigrid algorithm.

As a basis for $S_{p,\ell}(\Omega)$, we choose B-splines as
described by, e.g., de Boor \cite{DeBoor:Practical}.
To this end, we introduce an open knot vector
with the first and last knot repeated $p+1$ times each,
$
    (0,\ldots,0, h_\ell, 2h_\ell, \ldots, (n_\ell-1)h_\ell, 1,\ldots,1),
$
and define the normalized B-spline basis over this knot vector in the standard way.
We denote the B-spline basis functions by
$
	\{\varphi_{p,\ell}^{(1)},\ldots,\varphi_{p,\ell}^{(m_\ell)}\},
$
where $m_\ell=\dim S_{p,\ell}(\Omega) = n_\ell+p$.
They form a partition of unity, that is,
$\sum_{j=1}^{m_\ell} \varphi_{p,\ell}^{(j)}(x) = 1$ for all $x \in \Omega$.

In higher dimensions, we will assume $\Omega = (0,1)^d$ with $d>1$
and introduce tensor product B-spline basis functions
of the form $(x,y) \mapsto \varphi_{p,\ell}^{(j_1)}(x) \varphi_{p,\ell}^{(j_2)}(y)$
for two dimensions and analogous functions for higher dimensions.
The space spanned by these basis functions will again be denoted by $S_{p,\ell}(\Omega)$.
The extension to the case where
$\Omega \subset \mathbb R^d$ is the tensor product
of $d$ arbitrary bounded and open intervals is straightforward.
Similarly, it is no problem to have different spline degrees, grid sizes,
and/or number of subintervals for each dimension,
as long as the grid sizes are approximately equal in all directions.
However, for the sake of simplicity of the notation, we will always
assume that the discretization is identical in each coordinate direction.

\subsection{Model problem}\label{subsec:2b}

For the sake of simplicity, we restrict ourselves to the following model problem.
Let $\Omega = (0,1)^d$ and assume $f\in L^2(\Omega)$ to be a given function. Find
a function $u:\Omega \rightarrow \mathbb{R}$ such that
\begin{equation*}
			- \Delta u +u = f  \mbox{ in } \Omega,\qquad	\frac{\partial u}{\partial n} = 0  \mbox{ on } \partial\Omega.
\end{equation*}
In variational form, this problem reads:
find $u \in V:=H^1(\Omega)$ such that
\[
     \underbrace{( \nabla u, \nabla v )_{L^2(\Omega)} + (  u, v )_{L^2(\Omega)} }_{\displaystyle (u,v)_{A} :=}
    =
    \underbrace{( f, v )_{L^2(\Omega)} }_{\displaystyle  \langle  f, v \rangle := }
    \qquad \forall v \in V.
\]
Here and in what follows, $L^2(\Omega)$ is the standard Lebesgue space of square integrable
functions and $H^1(\Omega)$ denotes the standard Sobolev space of weakly differentiable functions
with derivatives in $L^2(\Omega)$.

Applying a Galerkin discretization using spline spaces,
we obtain the following discrete problem:
find $u_\ell \in V_{\ell} := S_{p,\ell}(\Omega) \subset V$ such that
\begin{equation*}
	( u_\ell, v_\ell)_{A} = \langle  f, v_\ell \rangle \qquad \forall v_\ell \in V_{\ell}.
\end{equation*}

Using the B-spline basis introduced in the previous subsection,
the discretized problem can be rewritten in matrix-vector notation,
\begin{equation}\label{eq:linear:system}
			A_\ell \ul{u}_\ell = \ul{f}_\ell,
\end{equation}
where $A_\ell$ is the B-spline stiffness matrix. Here and in what follows, underlined symbols for
primal variables, like $\ul{u}_\ell$ and $\ul{v}_\ell$, are the coefficient vectors representing
the corresponding functions $u_\ell$ and $v_\ell$ with respect to the B-spline basis. 

\begin{remark}\label{rem:1}
Our model problem may seem rather restrictive in the sense
that we only allow for tensor product domains.
In practical IgA problems, one often uses such domains as parameter domains
and introduces a geometry mapping, usually given in the same basis as used
for the discretization, which maps the parameter domain to the actual
physical domain of interest.
As long as the geometry mapping is well-behaved, a good solver on the
parameter domain can be used as a preconditioner for the problem on
the physical domain, and our model problem therefore captures all essential difficulties
that arise in the construction of multigrid methods in this more general setting.
Geometry mappings with singularities as well as multi-patch domains are beyond the
scope of this paper.
\end{remark}


\section{A multigrid solver framework}\label{sec:mg}

\subsection{Description of the multigrid algorithm}\label{subsec:3a}

We now introduce a standard multigrid algorithm for solving the discretized
equation~\eqref{eq:linear:system} on grid level~$\ell$.
Starting from an initial approximation~$\ul{u}^{(0)}_\ell$,
the next iterate $\ul{u}^{(1)}_\ell$ is obtained by the following two steps:
\begin{itemize}
    \item \emph{Smoothing procedure:}
        For some fixed number $\nu$ of smoothing steps, compute
              \begin{equation} \label{eq:sm}
                   \ul{u}^{(0,m)}_\ell := \ul{u}^{(0,m-1)}_\ell + \tau L_\ell^{-1}
                                    \left(\ul{f}_\ell -A_\ell\;\ul{u}^{(0,m-1)}_\ell\right)
                                    \qquad \mbox{for } m=1,\ldots,\nu,
              \end{equation}
              where $\ul{u}^{(0,0)}_\ell := \ul{u}^{(0)}_\ell$. The choice of
              the matrix $L_\ell$ and the damping parameter $\tau>0$ will be discussed below. 
    \item \emph{Coarse-grid correction:}
        \begin{itemize}
             \item Compute the defect and restrict it to grid level $\ell-1$ using a restriction matrix $I_\ell^{\ell-1}$:
                \[
                      \ul{r}_{\ell-1}^{(1)} := I_\ell^{\ell-1} \left(\ul{f}_\ell - A_\ell
                      \;\ul{u}^{(0,\nu)}_\ell\right).
                \]
             \item Compute the correction $\ul{p}_{\ell-1}^{(1)}$ by approximately solving the coarse-grid problem
                \begin{equation}\label{eq:coarse:grid:problem}
                    A_{\ell-1} \,\ul{p}_{\ell-1}^{(1)} =\ul{r}_{\ell-1}^{(1)}.
                \end{equation}
             \item Prolongate $\ul{p}_{\ell-1}^{(1)}$ to the grid level $\ell$ using a prolongation
                  matrix $I^\ell_{\ell-1}$ and add the result to the previous iterate:
                  \begin{equation}\nonumber
                       \ul{u}_{\ell}^{(1)} := \ul{u}^{(0,\nu)}_\ell +
                        I_{\ell-1}^\ell \, \ul{p}_{\ell-1}^{(1)}.
                  \end{equation}
        \end{itemize}
\end{itemize}
As we have assumed nested spaces, the intergrid transfer
matrices $I_{\ell-1}^\ell$ and $I_\ell^{\ell-1}$ can be chosen in
a canonical way:
$I_{\ell-1}^\ell$ is the canonical embedding and the restriction $I_\ell^{\ell-1}$
its transpose.
In the IgA setting, the prolongation matrix $I_{\ell-1}^\ell$ can be computed by means of
knot insertion algorithms.

If the problem \eqref{eq:coarse:grid:problem} on the coarser grid is solved exactly
(\emph{two-grid method}), the coarse-grid correction is given by
\begin{equation} \label{eq:method:cga}
        \ul{u}_\ell^{(1)} := \ul{u}_\ell^{(0,\nu)} +
        I_{\ell-1}^{\ell} \, A_{\ell-1}^{-1} \,  I_{\ell}^{\ell-1}
        \left( \ul{f}_\ell - A_\ell \;\ul{u}_\ell^{(0,\nu)}\right).
\end{equation}
In practice, the problem~\eqref{eq:coarse:grid:problem} is
approximately solved by recursively applying one step (\emph{V-cycle})
or two steps (\emph{W-cycle}) of the multigrid method. On
the coarsest grid level ($\ell=0$), the problem~\eqref{eq:coarse:grid:problem} is
solved exactly by means of a direct method.

The crucial remaining task is the choice of the
smoother. For multigrid methods for elliptic problems with a FEM discretization,
it is common to use a damped Jacobi iteration or a Gauss-Seidel iteration as the smoother.
This is also possible if an isogeometric discretization is used,
but leads to significant deterioration in the convergence rates
as $p$ is increased (\cite{GahalautEtAl:2013,HofreitherZulehner:2014c}).
This motivates to choose a non-standard smoother. In the sequel of this section,
we derive conditions on the matrix $L_\ell$, guaranteeing convergence, which
we use in the following two sections to construct a smoother for the model problem.

\subsection{Abstract multigrid convergence theory}\label{subsec:3b}

As the two-grid method is a linear iteration scheme, its action on the error is
fully described by the iteration matrix.
Let $\ul u_\ell^*$ denote the exact solution of \eqref{eq:linear:system}.
Then the initial error $\ul u_\ell^* - \ul u_\ell^{(0)}$
and the error after one two-grid cycle $\ul u_\ell^* - \ul u_\ell^{(1)}$ are
related by the equation
\[
    \ul u_\ell^* - \ul u_\ell^{(1)} =  (I-T_{\ell-1}) S_\ell^{\nu} \, (\ul u_\ell^* - \ul u_\ell^{(0)}),
\]
where
\[
	S_\ell = I - \tau L_\ell^{-1} A_\ell,
  \qquad
	T_{\ell-1} = I_{\ell-1}^{\ell} \, A_{\ell-1}^{-1} \,  I_{\ell}^{\ell-1} A_\ell.
\]
Observe that $S_\ell$ and $I - T_{\ell-1}$
are the iteration matrices of the smoother and
of the coarse-grid correction, respectively.

Throughout the paper we use the following notations.
For a symmetric and positive definite matrix $Q_\ell$ and
a vector $\ul{u}_\ell$, we define $\|\ul{u}_\ell\|_{Q_\ell} := (Q_\ell\ul{u}_\ell,\ul{u}_\ell)^{1/2} = \|Q_\ell^{1/2}\ul{u}_\ell\|$,
where $\|\cdot\|$ and $(\cdot,\cdot)$ are the Euclidean norm and scalar product, respectively.
The associated matrix norms are denoted by the same symbols $\|\cdot\|_{Q_\ell}$ and $\|\cdot\|$.
For ease of notation, we use the same symbols also for the norm of functions $u_\ell \in V_\ell$ and mean in
this case their application to the coefficient vector, i.e.,
$\| u_\ell \|_{Q_\ell} := \| \underline u_\ell \|_{Q_\ell}$.
As $A_\ell$ is the matrix representation of the bilinear form $(\cdot,\cdot)_A$, we have
$\|  u_\ell \|_{A_\ell} =\| u_\ell \|_{A}$, where
$\| \cdot \|_{A}:=(\cdot, \cdot )_{A}^{1/2}$ is the energy norm.

We assume that the matrix $L_\ell$ appearing in \eqref{eq:sm} is symmetric and positive definite.
Following the classical line of, e.g., \cite{Brenner:1996},
we show convergence in that norm, i.e., we show that
\begin{equation}\label{eq:norm:estim}
    q := \|(I-T_{\ell-1}) S_\ell^{\nu}\|_{L_\ell} < 1,
\end{equation}
which obviously implies the $q$-linear convergence property
\begin{equation*}
	\|\ul u_\ell^* - \ul u_\ell^{(1)}\|_{L_\ell} \le q \| \ul u_\ell^* - \ul u_\ell^{(0)} \|_{L_\ell}.
\end{equation*}
To analyze \eqref{eq:norm:estim}, we use semi-multiplicity of norms and obtain
\begin{align}
		 \| (I-T_{\ell-1}) S_\ell^{\nu}\|_{L_\ell}  &= \|L_\ell^{1/2} (I-T_{\ell-1}) S_\ell^{\nu}L_\ell^{-1/2} \| \nonumber\\
		& \le \|L_\ell^{1/2} (I-T_{\ell-1}) A_\ell^{-1} L_\ell^{1/2} \|\;\|L_\ell^{-1/2} A_\ell S_\ell^{\nu}L_\ell^{-1/2} \|.
			\label{eq:decomp}
\end{align}
Therefore, it suffices to prove the two conditions
\begin{align}
    \|L_\ell^{1/2} (I-T_{\ell-1}) A_\ell^{-1} L_\ell^{1/2} \|\ &\le C_A,
    \qquad \text{(approximation property)} \label{eq:ap}
    \\
    \|L_\ell^{-1/2} A_\ell S_\ell^{\nu}L_\ell^{-1/2} \| &\le C_S \nu^{-1},
    \qquad \text{(smoothing property)} \label{eq:sp}
\end{align}
cf.~\cite[eq.~(6.1.5)]{Hackbusch:1985} for this splitting.
Then \eqref{eq:norm:estim} follows immediately for $\nu>C_A C_S$, that is, if sufficiently
many smoothing steps are applied. As we are interested in robust convergence,
the constants $C_A$ and $C_S$ should be independent of both
the grid size and the spline degree.

The approximation property \eqref{eq:ap} is a direct consequence of the following
approximation error estimate.
Observe that $T_{\ell-1}$ is the matrix representation of the $A$-orthogonal projector from $V_\ell$ to $V_{\ell-1}$.
For simplicity, we use the same symbol $T_{\ell-1}$ also for the $A$-orthogonal projector from $V_\ell$ to $V_{\ell-1}$
itself and, later on, also for the $A$-orthogonal projector from the whole space $V = H^1(\Omega)$ to $V_{\ell-1}$.

\begin{lemma}\label{lem:mg-approx}
    The approximation error estimate
    \begin{equation}\label{eq:ap:2}
        \|  (I-T_{\ell-1})  u_\ell \|_{L_\ell}^2 \le C_A \|  u_\ell \|_{A}^2
        \qquad
        \forall  u_\ell \in V_\ell
    \end{equation}
    is equivalent to the approximation property \eqref{eq:ap} with the same constant $C_A$.
\end{lemma}
\begin{proof}
    The estimate (\ref{eq:ap:2}) can be rewritten as
    \begin{equation}\nonumber
		    \| X_\ell \| \le C_A^{1/2}
		    \quad \text{with} \quad X_\ell := L_\ell^{1/2} (I-T_{\ell-1}) A_\ell^{-1/2},
    \end{equation}
    which is equivalent to $\| X_\ell X_\ell^T \| \le C_A$.
    We have
    \[
        X_\ell X_\ell^T
        = L_\ell^{1/2} (I-T_{\ell-1}) A_\ell^{-1} (I-T_{\ell-1})^T L_\ell^{1/2}
        = L_\ell^{1/2} (I-T_{\ell-1})^2 A_\ell^{-1} L_\ell^{1/2},
    \]
    where we used $((I-T_{\ell-1}) A_\ell^{-1})^T = (I-T_{\ell-1}) A_\ell^{-1}$,
    which follows from $I-T_{\ell-1}$ being self-adjoint with respect to $A_\ell$.
    Since $I-T_{\ell-1}$ is a projector, the statement follows. \qed
\end{proof}

The following lemma states that the smoothing property is a direct consequence of
an inverse inequality.
\begin{lemma}\label{lem:smooth}
	Let $A_\ell$ and $L_\ell$ be symmetric and positive definite matrices.
  	Assume that the inequality
	\begin{equation}\label{eq:sp:2}
      \|u_\ell\|_{A}^2 \le C_I \|u_\ell\|_{L_\ell}^2 \qquad \forall u_\ell \in V_\ell
	\end{equation}
	holds. Then for $\tau\in(0,C_I^{-1}]$, the smoother~\eqref{eq:sm} satisfies the
	smoothing property~\eqref{eq:sp} with $C_S=\tau^{-1}$. Under the same assumptions,
	 $\|S_\ell\|_{A_\ell} \le 1 $ holds.
\end{lemma}
\begin{proof}
    The proof is based on \cite[Lemma~6.2.1]{Hackbusch:1985}.
    From~\eqref{eq:sp:2}, we immediately obtain $\tau^{-1} \ge \|A_\ell^{1/2} L_\ell^{-1/2} \|^2 = \rho(L_\ell^{-1/2} A_\ell L_\ell^{-1/2})$,
    where $\rho(\cdot)$ is the spectral radius.
    Observe that
    \[
        L_\ell^{-1/2} A_\ell S_\ell^{\nu} L_\ell^{-1/2}
				= L_\ell^{-1/2} A_\ell (I-\tau L_\ell^{-1} A_\ell)^{\nu} L_\ell^{-1/2}
				= \bar{A}_\ell (I-\tau \bar{A}_\ell)^{\nu}
    \]
    with $\bar{A}_\ell:=L_\ell^{-1/2} A_\ell L_\ell^{-1/2}$ and
    that $\bar{A}_\ell (I-\tau \bar{A}_\ell)^{\nu}$ is symmetric.
    Furthermore, we have the spectral radius bound
    $\rho(\tau \bar{A}_\ell) =\tau \rho( L_\ell^{-1/2} A_\ell L_\ell^{-1/2}) \le 1$.
    Thus,
		\begin{align*}
        &\| L_\ell^{-1/2} A_\ell S_\ell^{\nu} L_\ell^{-1/2} \| \\
				&\qquad = \rho(\bar{A}_\ell (I-\tau \bar{A}_\ell)^{\nu})
				= \sup_{\lambda\in \sigma( \bar{A}_\ell ) } \lambda (1-\tau \lambda)^{\nu}
				\le \tau^{-1} \sup_{\mu \in [0,1] } \mu (1-\mu)^{\nu} \\
				&\qquad= \tau^{-1} \left(\frac{\nu}{1 + \nu}\right)^{\nu} \frac{1}{1 + \nu}
				\le  \frac{\tau^{-1}}{ \nu},
		\end{align*}
    which shows \eqref{eq:sp}. Here $\sigma(\cdot)$ denotes the spectrum of a matrix. Similarly we have $\|S_\ell\|_{A_\ell} = \rho(I-\tau \bar{A}_\ell) \le 1 $.
\qed\end{proof}

Observe that estimate (\ref{eq:sp:2}) is, indeed, of the typical form of an inverse inequality,
if $L_\ell$ is a (properly scaled) mass matrix representing the $L^2$ inner product.

In view of Lemmas~\ref{lem:mg-approx} and~\ref{lem:smooth}, we can state sufficient conditions
for the convergence of a two-grid method.

\begin{theorem}\label{thrm:two:grid:gen1}
    Assume that there are constants $C_A$ and $C_I$, independent of the grid size and
    the spline degree, such that the approximation error estimate~\eqref{eq:ap:2}
    and the inverse inequality~\eqref{eq:sp:2} hold.
	Then the two-grid method converges for the choice $\tau\in(0, C_I^{-1}]$ and $\nu > \nu_0 := \tau^{-1}C_A$
	with rate $q = \nu_0/\nu$.
\end{theorem}
\begin{proof}
			The statement on the convergence of the two-grid method follows directly
			from~\eqref{eq:decomp}, Lemma~\ref{lem:mg-approx}, and Lemma~\ref{lem:smooth}.
\qed\end{proof}

The assumptions of Theorem~\ref{thrm:two:grid:gen1} are also sufficient to prove convergence of
a symmetrical variant of the W-cycle multigrid method.
Here, in addition to the smoothing steps before the coarse-grid correction (\textit{pre-smoothing}),
we perform an equal number of smoothing steps after the coarse-grid correction (\textit{post-smoothing}).

\begin{theorem}\label{thrm:two:grid:gen1a}
			Under the assumptions of Theorem~\ref{thrm:two:grid:gen1}, the W-cycle multi\-grid method with
			$\nu/2$ pre- and $\nu/2$ post-smoothing steps converges in the energy norm $\|\cdot\|_A$ for the choice
			$\tau\in(0, C_I^{-1}]$ and $\nu > 4 \nu_0 = 4 \tau^{-1}C_A$ with rate $q = 2\nu_0/\nu$.
\end{theorem}
\begin{proof}
			Theorem~\ref{thrm:two:grid:gen1} states that
			\begin{equation*}
								\|(I-T_{\ell-1}) S_\ell^{\nu}\|_{L_\ell} \le \frac{\tau^{-1}C_A}{\nu}.
			\end{equation*}
			As $S_\ell^{\nu/2}(I-T_{\ell-1}) S_\ell^{\nu/2}$ is self-adjoint in the scalar
			product $(\cdot,\cdot)_{A_\ell}$, this implies
			\begin{equation}\label{eq:cor:converg:1:pre}
			    \|S_\ell^{\nu/2}(I-T_{\ell-1}) S_\ell^{\nu/2}\|_{A_\ell}
			    =
			    \rho((I-T_{\ell-1}) S_\ell^{\nu})
			    \le
			    \|(I-T_{\ell-1}) S_\ell^{\nu}\|_{L_\ell}
			    \le
			    \frac{\tau^{-1}C_A}{\nu}.
			\end{equation}
			The iteration matrix of the W-cycle multigrid method is recursively given by
			\begin{equation*}
				W_\ell = S_\ell^{\nu/2} (I-I_{\ell-1}^\ell (I-W_{\ell-1}^2) A_{\ell-1}^{-1} I_\ell^{\ell-1} A_\ell) S_\ell^{\nu/2}
				\qquad \mbox{for } \ell > 0,
			\end{equation*}
			and $W_0 = 0$.
			Using the triangle inequality and semi-multiplicativity of norms, we obtain
			for the convergence rate
			\begin{align*}
				q_\ell & = \|W_\ell\|_{A_\ell}
					= \|S_\ell^{\nu/2} (I-I_{\ell-1}^\ell (I-W_{\ell-1}^2) A_{\ell-1}^{-1} I_\ell^{\ell-1} A_\ell) S_\ell^{\nu/2}\|_{A_\ell}\\
					& \le \|S_\ell^{\nu/2} (I-T_{\ell-1}) S_\ell^{\nu/2}\|_{A_\ell}
								+ \|S_\ell^{\nu/2} I_{\ell-1}^\ell W_{\ell-1}^2 A_{\ell-1}^{-1} I_\ell^{\ell-1} A_\ell S_\ell^{\nu/2}\|_{A_\ell}\\
					& \le \|S_\ell^{\nu/2} (I-T_{\ell-1}) S_\ell^{\nu/2}\|_{A_\ell}
								+ \|S_\ell\|_{A_\ell}^{\nu} \| A_\ell^{1/2} I_{\ell-1}^\ell A_{\ell-1}^{-1/2}\|^2 \|W_{\ell-1}\|_{A_{\ell-1}}^2.
			\end{align*}
			As the spaces $V_\ell$ are nested, we obtain $A_{\ell-1}=I^{\ell-1}_\ell A_\ell I_{\ell-1}^\ell$. Thus it follows
			$\| A_\ell^{1/2} I_{\ell-1}^\ell A_{\ell-1}^{-1/2}\|^2 = \rho( I^{\ell-1}_\ell A_\ell I_{\ell-1}^\ell A_{\ell-1}^{-1})=1$.
			Lemma~\ref{lem:smooth} states that $\|S_\ell\|_{A_\ell}\le 1$.
			Using these two statements, \eqref{eq:cor:converg:1:pre} and
			$q_{\ell-1} = \|W_{\ell-1}\|_{A_{\ell-1}}$, we obtain
			\[
				q_\ell \le \frac{\nu_0}{\nu} + q_{\ell-1}^2.
			\]
      Using the assumption $\nu \ge 4 \nu_0$, one shows by induction that
      $q_\ell \le 2\nu_0/\nu$.
\qed\end{proof}


\section{A robust multigrid method for one-dimensional domains}\label{sec:1d}

\subsection{Robust estimates for a subspace of the spline space}

For the model problem, we have $V=H^1(\Omega)$ and
\begin{equation}\label{eq:model:v}
			\|\cdot\|_{A}= \|\cdot\|_{H^1(\Omega)}.
\end{equation}
The standard multigrid convergence analysis as introduced by Hackbusch~\cite{Hackbusch:1985}
gives convergence in a (properly scaled) $L^2$-norm,
\begin{equation}\label{eq:plain:mass}
			\|\cdot\|_{L_\ell} = h_\ell^{-1} \|\cdot\|_{L^2(\Omega)},
\end{equation}
and thus the choice $L_\ell = h_\ell^{-2} M_\ell$, where 
$M_\ell$ is the mass matrix, consisting of pairwise $L^2$-scalar products of the basis functions.

To show robust convergence of the multigrid solver with a Jacobi smoother,
one would show that the diagonal of $A_\ell$ and $h_\ell^{-2} M_\ell$ are spectrally equivalent.
This equivalence holds robustly in the grid size,
however the involved constants deteriorate with increasing spline degree $p$.
This issue is closely related to the so-called condition number of the
B-spline basis (see, e.g., \cite{DeBoor:1972}).
The growth of the condition number with $p$ explains why standard Jacobi iteration
and Gauss-Seidel iteration do not work well for B-splines.

The aforementioned equivalence is not necessary for a direct application of the choice
$L_\ell = h_\ell^{-2} M_\ell$, which is the mass-Richardson smoother
already studied in~\cite{HofreitherZulehner:2014b}.
Local Fourier analysis indicates that this smoother should lead to robust
convergence of the multigrid solver. However, the numerical experiments
in~\cite{HofreitherZulehner:2014b} show that this is not the case
and iteration numbers still deteriorate with $p$. The reason for this effect is motivated as follows.

With the choice~\eqref{eq:model:v} and~\eqref{eq:plain:mass},
the condition~\eqref{eq:sp:2} reads
\begin{equation}
    \label{eq:inveq}
		 \|u_\ell\|_{H^1(\Omega)} \le C_I^{1/2}  h_\ell^{-1} \|u_\ell\|_{L^2(\Omega)}
		 \qquad \forall u_\ell \in V_{\ell}=S_{p,\ell}(\Omega).
\end{equation}
In other words, it is required that the spline space satisfies an inverse inequality with a constant
that is independent of the grid size and the spline degree.
However, such a robust inverse inequality does not hold,
as the counterexample given in \cite{Takacs:Takacs:2015} shows:
for each $p\in \mathbb{N}$ and each $\ell\in \mathbb{N}_0$, there exists a spline
$w_\ell \in S_{p,\ell}(\Omega)$ with
\begin{equation}
    \label{eq:counterex}
    |w_\ell|_{H^1(\Omega)}  \ge p {h_\ell^{-1}\|w_\ell\|_{L^2(\Omega)}}.
\end{equation}
As we have to choose $\tau \le C_I^{-1}$ for the smoother not to diverge, the
existence of such splines $w_\ell$ implies that one has to choose $\tau = \mathcal{O}(p^{-2})$, 
which causes the convergence rates to deteriorate.
The splines $w_\ell$ in the counterexample are non-periodic
and therefore cannot be captured by local Fourier analysis and similar tools which
consider only the periodic setting.

In~\cite{Takacs:Takacs:2015} it was shown that a robust inverse estimate of the form \eqref{eq:inveq}
does hold for the following large subspace of $S_{p,\ell}(\Omega)$.

\begin{definition}
    \label{def:S-tilde}
    Let $\Omega=(0,1)$. We denote by $\widetilde{S}_{p,\ell}(\Omega)$ the space of all
    $u_\ell\in S_{p,\ell}(\Omega)$ whose odd derivatives of order less than $p$ vanish at the boundary,
	\[
        \frac{\partial^{2l+1}}{\partial x^{2l+1}} u_\ell(0)=\frac{\partial^{2l+1}}{\partial x^{2l+1}} u_\ell(1) = 0
        \quad \text{ for all } l \in \mathbb{N}_0 \text{ with } 2l+1 < p.
    \]
\end{definition}
The space $\widetilde{S}_{p,\ell}(\Omega)$ is almost as large as $S_{p,\ell}(\Omega)$:
their dimensions differ by $p$ (for $p$ even) or $p-1$ (for $p$ odd).

\begin{theorem}[\cite{Takacs:Takacs:2015}]
    \label{thrm:inverse}
    For all spline degrees $p\in \mathbb{N}$ and all grid levels $\ell\in\mathbb{N}_0$,
    we have the inverse inequality
    \[
        |u_\ell|_{H^1(\Omega)} \le 2 \sqrt{3} h_\ell^{-1} \|u_\ell \|_{L^2(\Omega)}
        \qquad \forall u_\ell \in \widetilde{S}_{p,\ell}(\Omega).
    \]
\end{theorem}

This result makes it clear that the
counterexample~\eqref{eq:counterex} describes the effect of only a few outliers
(see also \cite{Cottrell:2006} on the topic of spectral outliers in IgA), connected to the boundary.
We make use of this fact by constructing a mass-Richardson smoother with a (low-rank) boundary correction.

Using $\|u_\ell\|_A^2 = |u_\ell|_{H^1(\Omega)}^2 + \|u_\ell\|_{L^2(\Omega)}^2$ and $h_\ell \le 1$,
the inverse inequality implies
\begin{equation}\label{eq:our:inverse}
    \|u_\ell\|_A \le c \|u_\ell\|_{h_{\ell}^{-2} M_\ell}
    \qquad \forall u_\ell \in \widetilde{S}_{p,\ell}(\Omega).
\end{equation}
Here and in the sequel, we use $c$ to refer to a generic constant
which is independent of both the grid level $\ell$ and the spline degree $p$.

In addition to the inverse estimate, also an approximation property
for the subspace $\widetilde{S}_{p,\ell}(\Omega)$ was proved in \cite{Takacs:Takacs:2015}.
We slightly refine the result here as follows.

\begin{theorem}\label{thm:ap}
		We have the approximation error estimate
		\begin{equation*}
				\|(I-\widetilde{T}_{\ell}) u \|_{L^2(\Omega)} \le 2 \sqrt{2} h_{\ell} \|u\|_{H^1(\Omega)}
				\qquad \forall u \in H^1(\Omega),
		\end{equation*}
    where
    $\widetilde{T}_{\ell}: H^1(\Omega) \to \widetilde{S}_{p,\ell}(\Omega)$
    is the $A$-orthogonal projector.
\end{theorem}

The proof for this theorem is given in the Appendix.

Using \eqref{eq:our:inverse} and Theorem~\ref{thm:ap}, we immediately obtain
assumptions \eqref{eq:sp:2} and \eqref{eq:ap:2} of Theorem~\ref{thrm:two:grid:gen1}
with the choice $V_\ell=\widetilde{S}_{p,\ell}(\Omega)$ and $L_\ell = h_\ell^{-2} M_\ell$.
Thus, a two-grid method using the mass smoother would be robust in this subspace.
However, the problems we are interested in are typically discretized
in the full spline space $S_{p,\ell}(\Omega)$.
We extend the choice of $L_\ell$ robustly to such problems in the following.

\subsection{Extension of the inverse inequality to the entire spline space}

In this section, we modify the mass matrix in such a way as to satisfy
a robust inverse inequality in the entire spline space.
This is done by means of an additional term which is
essentially a discrete harmonic extension.
We set
\[
    \widetilde L_\ell := h_\ell^{-2} M_\ell +  (I-\widetilde T_{\ell})^T A (I-\widetilde T_{\ell}).
\]
Here, the second term corrects for the violation of the inverse inequality in the $A$-orthogonal complement
of $\widetilde S_{p,\ell}(\Omega)$ by incorporating the original operator $A$.
This term satisfies the energy minimization property
\[
	\|  u_\ell \|_{(I - \widetilde T_\ell)^T A (I - \widetilde T_\ell) }
    = \| (I - \widetilde T_\ell) u_\ell \|_A
    = \inf_{v_\ell \in \widetilde S_{p,\ell}(\Omega)} \| u_\ell + v_\ell \|_A.
\]

We now prove the inverse inequality for the modified norm $\|\cdot\|_{\widetilde L_\ell}$.
The proof requires both the inverse inequality and the approximation error estimate
in $\widetilde S_{p,\ell}(\Omega)$.

\begin{lemma}
    \label{lem:ssm-SP}
    We have the inverse inequality
    \[
        \| u_\ell \|_{A} \le c \| u_\ell \|_{\widetilde L_\ell}
        \qquad \forall u_\ell \in S_{p,\ell}(\Omega)
    \]
    with a constant $c$ which is independent of $\ell$ and $p$.
\end{lemma}
\begin{proof}
    Let $u_\ell \in S_{p,\ell}(\Omega)$.
    Using the triangle inequality, the inverse inequality \eqref{eq:our:inverse} and again
    the triangle inequality  we obtain
    \begin{align*}
        \| u_\ell \|_A
        &\le \| \widetilde T_\ell u_\ell \|_A + \| (I-\widetilde T_\ell) u_\ell  \|_A 
         \le c h_\ell^{-1} \| \widetilde T_\ell u_\ell \|_{M_\ell} + \| (I-\widetilde T_\ell) u_\ell \|_A \\
        &\le c h_\ell^{-1} \| u_\ell \|_{M_\ell} + c h_\ell^{-1} \| (I-\widetilde T_\ell) u_\ell \|_{M_\ell} 
        + \| (I-\widetilde T_\ell) u_\ell \|_A.
    \end{align*}
    By Theorem~\ref{thm:ap}, we have
    $
        \| (I - \widetilde T_\ell) u_\ell  \|_{M_\ell} =
        \| (I - \widetilde T_\ell)^2 u_\ell  \|_{M_\ell} \le
        c h_\ell \| (I-\widetilde T_\ell) u_\ell \|_{A}
    $,
    and it follows
    \[
        \| u_\ell \|_A
        \le c \left( h_\ell^{-1} \| u_\ell \|_{M_\ell} + \| (I-\widetilde T_\ell) u_\ell \|_{A} \right).
    \]
    The statement follows since
    $\| (I-\widetilde T_\ell) u_\ell \|_A =
    \| u_\ell\|_{(I - \widetilde T_\ell)^T A_\ell (I - \widetilde T_\ell)}$.
    \qed
\end{proof}


\subsection{A robust multigrid method for the whole spline space}

We first show that Theorem~\ref{thm:ap} can be easily extended to the projection into
the entire spline space.

\begin{lemma}\label{lem:approx2}
    The approximation error estimate
    \[
        \|(I-T_{\ell}) u\|_{L^2(\Omega)}  \le c h_\ell \|u\|_{H^1(\Omega)}
        \qquad\forall u \in H^1(\Omega),
    \]
    holds, where $T_{\ell}: H^1(\Omega) \to S_{p,\ell}(\Omega)$ is the $A$-orthogonal projector.
\end{lemma}
\begin{proof}
    As $\widetilde{S}_{p,\ell}(\Omega) \subset S_{p,\ell}(\Omega)$,
    we have $\widetilde{T}_{\ell} = \widetilde{T}_{\ell} T_{\ell}$.
    This identity, the triangle inequality, 
    Theorem~\ref{thm:ap}, and the stability of the $A$-orthogonal projector $T_\ell$ yield
    \begin{align*}
        &\|(I-T_{\ell}) u\|_{L^2(\Omega)}^2
        \le 2( \|(I-\widetilde{T}_{\ell}) u\|_{L^2(\Omega)}^2 + \|(I-\widetilde{T}_{\ell})T_{\ell} u\|_{L^2(\Omega)}^2)\\
        &\qquad\quad\quad\le  2 c h_\ell  (  \| u \|_{H^1(\Omega)}^2 + \| T_{\ell} u\|_{H^1(\Omega)}^2 ) 
        \le 4 c h_\ell \| u \|_{H^1(\Omega)}^2\quad \forall u \in H^1(\Omega).
        \qedhere
    \end{align*}
\end{proof}

The robust inverse inequality (Lemma~\ref{lem:ssm-SP}) together with the approximation
property (Lemma~\ref{lem:approx2})
allow us to prove robust convergence of the following multigrid method for the
space $V_\ell:= S_{p,\ell}(\Omega)$.
It turns out that we also obtain robust convergence if we do not
project into $\widetilde{S}_{p,\ell}(\Omega)$, but into a subspace
$S^I_{p,\ell}(\Omega)$, which may be easier to handle in practice.

\begin{theorem}
    \label{thm:inner-smoother}
    Let $S^I_{p,\ell}(\Omega) \subseteq \widetilde{S}_{p,\ell}(\Omega)$ a subspace and
    $T^I_\ell: S_{p,\ell}(\Omega) \to S^I_{p,\ell}(\Omega)$ the $A$-orthogonal projector.
    Consider the smoother~\eqref{eq:sm} with the choice
    \begin{equation}\label{eq:inner-smoother}
        L_\ell := h_{\ell}^{-2} M_\ell + (I-T^I_\ell)^T A_\ell (I-T^I_\ell).
    \end{equation}
    Then there exists a damping parameter $\tau>0$ and a choice of $\nu_0>0$, both independent of $\ell$ and
    $p$, such that the two-grid method with $\nu>\nu_0$
    smoothing steps converges with rate $q=\nu_0/\nu$.

    With the same choice of $\tau$ and with $\nu > 4 \nu_0$, also the W-cycle multigrid
    method with $\nu/2$ pre- and $\nu/2$ post-smoothing steps converges in the energy norm with
    convergence rate $q=2\nu_0/\nu$.
\end{theorem}
\begin{proof}
    We show the assumptions of Theorem~\ref{thrm:two:grid:gen1}.
    The statements then follow from Theorems~\ref{thrm:two:grid:gen1} and~\ref{thrm:two:grid:gen1a}.

    \textbf{Proof of assumption~\eqref{eq:sp:2}.}
    Let $ u_\ell \in S_{p,\ell}(\Omega)$.  Lemma~\ref{lem:ssm-SP} states
    \[
        \|u_\ell\|_{A}^2
        \le
        c (\|u_\ell\|_{h_{\ell}^{-2} M_\ell}^2 + \| (I-\widetilde{T}_\ell)u_\ell\|_{A}^2).
    \]
    Using the energy-minimizing property of the projection operators, we obtain
    \begin{align*}
        \| (I-\widetilde{T}_\ell)u_\ell\|_{A} & =
        \inf_{v_\ell\in \widetilde{S}_{p,\ell}(\Omega) } \| u_\ell-v_\ell\|_{A} \\
        & \le \inf_{v_\ell\in S^I_{p,\ell}(\Omega) } \| u_\ell-v_\ell\|_{A} = \| (I-T^I_\ell)u_\ell\|_{A}
    \end{align*}
    since $S^I_{p,\ell}(\Omega) \subseteq \widetilde{S}_{p,\ell}(\Omega)$.
    Combining these estimates, we get the desired statement
    \begin{equation}\label{eq:sp:6}
        \|u_\ell\|_{A}^2 \le c \| u_\ell\|_{L_\ell}^2.
    \end{equation}

    \textbf{Proof of assumption~\eqref{eq:ap:2}.}
    Let $ u_\ell \in S_{p,\ell}(\Omega)$. Lemma~\ref{lem:approx2} implies
    \[
        \|(I-T_{\ell-1}) u_\ell\|_{ h_\ell^{-2} M_\ell }^2 \le c \|u_\ell\|_{ A }^2
    \]
    Since both $I - T^I_{\ell}$ and $I - T_{\ell-1}$ are $A$-orthogonal projectors
    and thus are stable in the $A$-norm, we obtain
    the approximation error estimate~\eqref{eq:ap:2} via
    \begin{align*}
        \| (I - T_{\ell-1}) u_\ell \|_{L_\ell}^2
        &= \| (I - T_{\ell-1}) u_\ell \|_{ h_\ell^{-2} M_\ell }^2
          +	\| (I - T^I_{\ell}) (I - T_{\ell-1}) u_\ell \|_{A}^2 \\
       & \le (c + 1) \|  u_\ell \|_{A}^2 .
        \qedhere
    \end{align*}
\end{proof}

The choice $S^I_{p,\ell}(\Omega) = \widetilde{S}_{p,\ell}(\Omega)$ in the above theorem
is admissible and results in a robust multigrid method.
However, the space $\widetilde{S}_{p,\ell}(\Omega)$ is somewhat difficult to work with
in practice since it is not easy to represent in terms of the standard B-spline basis.
Therefore, we choose a slightly smaller space for $S^I_{p,\ell}(\Omega)$ which is based
on a simple splitting of the degrees of freedom.
To this end, we split the spline space $S_{p,\ell}(\Omega)$ into ``boundary'' and ``inner''
functions,
\[
			S_{p,\ell}(\Omega) = S_{p,\ell}^{\Gamma}(\Omega) + S_{p,\ell}^{I}(\Omega),
\]
where $S_{p,\ell}^{\Gamma}(\Omega)$ is spanned by the first $p$ B-spline basis functions
$\varphi_{p,\ell}^{(1)},\ldots, \varphi_{p,\ell}^{(p)}$
and the last $p$ basis functions
$\varphi_{p,\ell}^{(m_\ell-p+1)},\ldots, \varphi_{p,\ell}^{(m_\ell)}$, whereas
the space $S_{p,\ell}^{I}(\Omega)$ is spanned by all the remaining basis functions.
By construction, $S_{p,\ell}^{I}(\Omega)$ consists of all splines that vanish on the boundary
together with all derivatives up to order $p-1$.
Therefore, we have
$S_{p,\ell}^{I}(\Omega) \subseteq \widetilde{S}_{p,\ell}(\Omega)$.

Due to this subspace relation, the inverse inequality \eqref{eq:our:inverse}
remains valid in $S_{p,\ell}^{I}(\Omega)$.
However, the analogue of the approximation property Theorem~\ref{thm:ap} does not
in general hold in this smaller space.
The existence of the larger subspace $\widetilde{S}_{p,\ell}(\Omega)$ in which both
properties hold is crucial for the proof of assumption \eqref{eq:sp:2} above,
even though we do not directly use the space $\widetilde{S}_{p,\ell}(\Omega)$ in practice.

Concerning the practical realization of such a smoother, observe that we can
reorder the degrees of freedom based on the splitting
$S_{p,\ell}(\Omega) =S_{p,\ell}^{\Gamma}(\Omega) +S_{p,\ell}^{I}(\Omega) $
and write the matrix $A_\ell$ and the vector $\ul{u}_\ell$ in block structure as
\begin{equation}\label{eq:k:blocks}
			A_\ell = \begin{pmatrix}
                A_{\Gamma\Gamma,\ell} & A_{I\Gamma,\ell}^T \\
                A_{I \Gamma,\ell}     & A_{II,\ell}
            \end{pmatrix},
            \qquad
			\ul{u}_\ell = \begin{pmatrix}
                \ul{u}_{\Gamma,\ell} \\ \ul{u}_{I,\ell}
            \end{pmatrix}.
\end{equation}
Using a matrix representation of the projector $T^I_\ell$,
direct computation yields
\[
    (I - T^I_\ell) \ul u_\ell =
    \begin{pmatrix}
        \ul u_{\Gamma,\ell} \\ -A_{II,\ell}^{-1} A_{I \Gamma,\ell} \ul u_{\Gamma,\ell}
    \end{pmatrix},
\]
and therefore the matrix $L_\ell$ can be represented using the Schur complement as
\begin{equation}\label{eq:def:khat}
    L_\ell =
    h_\ell^{-2} M_\ell + C_\ell :=
    h_\ell^{-2} M_\ell +
        \begin{pmatrix}
            A_{\Gamma\Gamma,\ell} - A_{I\Gamma,\ell}^T A_{II,\ell}^{-1} A_{I\Gamma,\ell} & 0 \\
            0 & 0
        \end{pmatrix}.
\end{equation}
We remark that $\ul u_\ell^T C_\ell \ul u_\ell$ represents the (squared) energy of the
discrete harmonic extension of $\ul u_{\Gamma,\ell}$ from $S^\Gamma_{p,\ell}(\Omega)$ to $S_{p,\ell}(\Omega)$.


\section{A robust multigrid method for two-dimensional domains}\label{sec:2d}

In this section, we extend the theory presented in the previous section to the (more relevant) case of
two-dimensional domains. As outlined in Remark~\ref{rem:1}, we restrict ourselves to problems without
geometry mapping, i.e., problems on the unit square only. However, these approaches can be used as
preconditioners for problems on general geometries if there is a regular geometry mapping.

In the following, we are interested in solving the problem \eqref{eq:linear:system} with $d=2$. Here,
we assume to have a tensor product space. Without loss of generality, we assume for sake of simplicity
that $\Omega=(0,1)^2$ and
\begin{equation*}
		\mathcal{S}_{p,\ell}(\Omega) := S_{p,\ell}(0,1) \otimes S_{p,\ell}(0,1).
\end{equation*}

Below, we will use calligraphic letters to refer to matrices for the two-dimensional domain,
whereas standard letters refer to matrices for the one-dimensional domain.
Using the tensor product structure of the problem and the discretization,
the mass matrix has the tensor product structure
\begin{equation*}
		\mathcal{M}_\ell = M_\ell \otimes M_\ell,
\end{equation*}
where $\otimes$ denotes the Kronecker product.
To keep the notation simple, we assume that we have the same basis and thus
the same matrix $M_\ell$ for both directions of the two-dimensional domain.
However, this is not needed for the analysis.

The stiffness matrix in two dimensions has the representation
\begin{equation}\label{eq:decomp:k:twod}
		\mathcal{A}_\ell = K_\ell \otimes M_\ell + M_\ell \otimes K_\ell + M_\ell \otimes M_\ell,
\end{equation}
where $K_\ell$ represents the scalar product $(\nabla u_\ell,\nabla v_\ell)_{L^2(\Omega)}$ in one
dimension. This decomposition reflects that \[( u_\ell,v_\ell)_A = (\partial_x u_\ell,\partial_x v_\ell)_{L^2(\Omega)}
+(\partial_y u_\ell,\partial_y v_\ell)_{L^2(\Omega)}+( u_\ell, v_\ell)_{L^2(\Omega)}.\]
The problem of interest is now
\[
    \mathcal{A}_\ell \ul u_\ell = \ul{f}_\ell.
\]
Here, the multigrid framework from Section~\ref{sec:mg} applies unchanged, only
replacing standard letters by calligraphic letters.
Again, we have to choose a suitable smoother by determining the matrix $\mathcal{L}_\ell$.
To this end, we again have a look at the approximation error estimate
and the inverse inequality.

Let $\widetilde{\mathcal{S}}_{p,\ell}(\Omega) := \widetilde{S}_{p,\ell}(0,1) \otimes \widetilde{S}_{p,\ell}(0,1)$
and let $\|\cdot\|_{\mathcal{A}}$ be the energy norm for the problem in two
dimensions, i.e., such that $\| u_\ell \|_{\mathcal{A}}=\|\underline  u_\ell\|_{\mathcal{A}_\ell}$ holds for
all $u_\ell \in \mathcal{S}_{p,\ell}(\Omega) $.
An approximation error estimate for the space $\widetilde{\mathcal{S}}_{p,\ell}(\Omega)$
was shown in \cite[Theorem~8]{Takacs:Takacs:2015}.
However, there, not the $\mathcal{A}$-orthogonal projector was analyzed. The following
statement is analogous to Theorem~\ref{thm:ap} and can be proved based on the
one-dimensional result.

\begin{theorem}\label{thrm:approx:2d}
    The approximation error estimate
	\begin{equation*}
		\| (I-\widetilde{\mathcal T}_{\ell}) u_\ell \|_{L^2(\Omega)} \le c h_\ell \|u_\ell\|_{\mathcal{A}}
			\qquad \forall u_\ell\in \mathcal{S}_{p,\ell}(\Omega),
	\end{equation*}
    where $\widetilde{\mathcal{T}}_\ell: \mathcal{S}_{p,\ell}(\Omega) \to \widetilde{\mathcal{S}}_{p,\ell}(\Omega)$
    is the $\mathcal{A}$-orthogonal projector,
	holds with a constant $c$ which is independent of $\ell$ and $p$.
\end{theorem}
\begin{proof}
       Theorem~\ref{thm:ap} implies that
		\begin{equation}\nonumber
				\|  (I- \widetilde{T}_\ell) u_\ell \|_{ M_\ell} \le ch_\ell \|u_\ell\|_{A_\ell}
						\qquad \forall u_\ell\in S_{p,\ell}(0,1),
		\end{equation}
		where $\widetilde{T}_\ell: S_{p,\ell}(0,1) \to \widetilde{S}_{p,\ell}(0,1)$ is the
		$A$-orthogonal projector.
		Stability of the $A$-orthogonal projector means that
		\[
				\|  (I- \widetilde{T}_\ell) u_\ell \|_{A_\ell} \le \|u_\ell\|_{A_\ell}
						\qquad \forall u_\ell\in S_{p,\ell}(0,1).
    \]
		Combining these statements, we obtain
		\[
				\|  (I- \widetilde{T}_\ell\otimes \widetilde{T}_\ell) u_\ell \|_{A_\ell\otimes M_\ell+M_\ell \otimes A_\ell}
							\le c h_\ell \|u_\ell\|_{A_\ell\otimes A_\ell}
							\qquad \forall u_\ell\in  \mathcal{S}_{p,\ell}(\Omega) .
    \]
    Note that $A_\ell = K_\ell+M_\ell$. Therefore, $\mathcal{A}_\ell =K_\ell \otimes M_\ell + M_\ell \otimes K_\ell 
    + M_\ell \otimes M_\ell \le A_\ell\otimes M_\ell+M_\ell\otimes A_\ell $ holds, where $\le$ is understood in the
    spectral sense. Moreover,
    \begin{align*}
            \mathcal{A}_\ell \mathcal{M}_\ell^{-1} \mathcal{A}_\ell 
            & = K_\ell M_\ell^{-1}K_\ell\otimes M_\ell
             + K_\ell\otimes K_\ell 
             + K_\ell\otimes M_\ell 
             + K_\ell\otimes K_\ell 
              \\
            &\quad + M_\ell\otimes K_\ell M_\ell^{-1}K_\ell
            + M_\ell  \otimes K_\ell 
            + K_\ell\otimes M_\ell
            + M_\ell  \otimes K_\ell
            +  M_\ell \otimes M_\ell \\
            & \ge K_\ell \otimes K_\ell + K_\ell \otimes M_\ell + M_\ell \otimes K_\ell + M_\ell \otimes M_\ell
            = A_\ell\otimes A_\ell
    \end{align*}
    holds. Using these two estimates, we obtain further 
		\[
				\| (I- \widetilde{T}_\ell\otimes \widetilde{T}_\ell) u_\ell \|_{\mathcal{A}_\ell}
							\le c h_\ell \|u_\ell\|_{\mathcal{A}_\ell \mathcal{M}_\ell^{-1} \mathcal{A}_\ell}
							\qquad \forall u_\ell\in  \mathcal{S}_{p,\ell}(\Omega) .
    \]
		As the $\mathcal{A}$-orthogonal projection minimizes the norm 
		$\|\cdot\|_{\mathcal{A}}=\|\cdot\|_{\mathcal{A}_\ell}$, we obtain
		\[
				\|  (I-\widetilde{\mathcal T}_{\ell}) u_\ell \|_{\mathcal{A}_\ell}
							\le c h_\ell \|u_\ell\|_{\mathcal{A}_\ell \mathcal{M}_\ell^{-1} \mathcal{A}_\ell}
							\qquad \forall u_\ell\in  \mathcal{S}_{p,\ell}(\Omega) ,
    \]
    or in matrix form
		\[
				\| \mathcal{A}_\ell^{1/2} (I-\widetilde{\mathcal T}_{\ell})  \mathcal{A}_\ell^{-1}  \mathcal{M}_\ell^{1/2} \|
							\le c h_\ell.
    \]
		By transposing and using $\mathcal{A}_\ell^{-1}\widetilde{\mathcal T}_{\ell}^T 
		=  \widetilde{\mathcal T}_{\ell}\mathcal{A}_\ell^{-1}$,
		it follows
		\begin{equation}\nonumber
				\| \mathcal{M}_\ell^{1/2} (I-\widetilde{\mathcal T}_{\ell}) \mathcal{A}_\ell^{-1/2} \|
							\le c h_\ell
		\end{equation}
        which can be equivalently rewritten as the desired result
		\[
            \| (I-\widetilde{\mathcal T}_{\ell}) u_\ell \|_{\mathcal{M}_\ell} \le c h_\ell \|u_\ell\|_{\mathcal{A}_\ell}
            \qquad \forall u_\ell\in  \mathcal{S}_{p,\ell}(\Omega) .
            \qedhere
        \]
\end{proof}

Completely analogously to the proof of Lemma~\ref{lem:approx2}, we can extend this
approximation error estimate to
\begin{equation}
    \label{eq:ap:2d}
  \|(I-\mathcal{T}_{\ell-1}) u_\ell \|_{L^2(\Omega)} \le c h_\ell \|u_\ell \|_{ \mathcal{A} }
  \qquad\forall u_\ell  \in \mathcal S_{p,\ell}(\Omega),
\end{equation}
where
$\mathcal{T}_\ell: H^1(\Omega) \to \mathcal{S}_{p,\ell}(\Omega)$
is the $\mathcal{A}$-orthogonal projector.

From \cite[Theorem~9]{Takacs:Takacs:2015}, the following robust inverse
inequality in $\widetilde{\mathcal{S}}_{p,\ell}(\Omega)$, analogous to the
one-dimensional result \eqref{eq:our:inverse}, follows:
\[
			\|u_\ell\|_{\mathcal{A}} \le 2\sqrt{6} h_{\ell}^{-1} \|u_\ell\|_{L^2(\Omega)} \qquad 
			\forall u_\ell \in \widetilde{\mathcal{S}}_{p,\ell}(\Omega).
\]
As in Section~\ref{sec:1d} for the one-dimensional case,
we could set up a smoother based on the $\mathcal A$-orthogonal projection
to an interior space $S_{p,\ell}^I(0,1) \otimes S_{p,\ell}^I(0,1)$ and prove,
by completely analogous arguments, that the resulting multigrid method is robust.
However, the realization of this smoother requires the discrete harmonic extension
from a boundary layer of width $p$ to the interior, which is too computationally expensive.

Instead we propose a slightly modified smoother which better exploits the tensor
product structure of the underlying spaces.
Using $A_\ell=K_\ell+M_\ell$, we obtain
\[
		\mathcal{A}_\ell \le A_\ell \otimes M_\ell + M_\ell \otimes A_\ell.
\]
As we have seen that $A_\ell\le c L_\ell$ for the choice of $L_\ell$
introduced in Section~\ref{sec:1d},
the inverse inequality would be satisfied for the choice
\[
		L_\ell \otimes M_\ell + M_\ell \otimes L_\ell,
\]
which can be expressed using $L_\ell = h_\ell^{-2} M_\ell + C_\ell$ from \eqref{eq:def:khat}
as follows:
\[
			2 h_\ell^{-2} M_\ell \otimes M_\ell + C_\ell \otimes M_\ell + M_\ell \otimes C_\ell
\]
Instead we choose $\mathcal{L}_\ell$ to be a slight perturbation of that matrix, namely,
\[
		\mathcal{L}_\ell := h_\ell^{-2}  M_\ell \otimes M_\ell + C_\ell \otimes M_\ell + M_\ell \otimes C_\ell
					 = h_\ell^2 ( L_\ell \otimes L_\ell - C_\ell\otimes C_\ell ),
\]
Thus, $\mathcal{L}_\ell$ is the sum of a tensor product matrix
$ h_\ell^2 L_\ell \otimes L_\ell$ and a low-rank correction $ h_\ell^2 C_\ell\otimes C_\ell$.
The following two lemmas show the assumptions for multigrid convergence
when using the smoother based on this $\mathcal{L}_\ell$.

\begin{lemma}\label{lem:2d-inverse-mod}
		There is a constant $c$, independent of $\ell$ and $p$,
		such that $\mathcal{A}_\ell \le c\, \mathcal{L}_\ell $.
\end{lemma}
\begin{proof}
		The one-dimensional inverse inequality~\eqref{eq:sp:6} in matrix formulation reads
        $
					A_\ell \le c L_\ell.
        $
		Using $L_\ell = h_\ell^{-2} M_\ell+C_\ell$, we obtain
        \begin{align*}
            \mathcal{A}_\ell & \le  A_\ell \otimes M_\ell + M_\ell \otimes A_\ell
            \le c ( L_\ell \otimes M_\ell + M_\ell \otimes L_\ell )  \\
            & = c ( 2 h_\ell^{-2} M_\ell\otimes M_\ell+ C_\ell \otimes M_\ell + M_\ell \otimes C_\ell )   \\
            & \le 2c (   h_\ell^{-2}M_\ell\otimes M_\ell+ C_\ell \otimes M_\ell + M_\ell \otimes C_\ell )
            = 2 c\mathcal{L}_\ell.
            \qedhere
        \end{align*}
\end{proof}

\begin{lemma}
    \label{lem:ap:2d}
    The approximation property
    \[
       \| (I-\mathcal{T}_{\ell-1}) u_\ell \|_{\mathcal{L}_\ell}
       \le c \| u_\ell \|_{\mathcal{A}}
       \qquad \forall u_\ell \in \mathcal S_{p,\ell}(\Omega)
    \]
    holds with a constant $c$ which is independent of $\ell$ and $p$.
\end{lemma}
\begin{proof}
   By an energy minimization argument, we know that $C_\ell \le A_\ell$ and therefore
   \[
       \mathcal{L}_\ell \le h_\ell^{-2} \mathcal M_\ell + M_\ell \otimes A_\ell + A_\ell \otimes M_\ell.
   \]
   Furthermore, since
   $
       M_\ell \otimes A_\ell + A_\ell \otimes M_\ell = \mathcal A_\ell + \mathcal M_\ell
   $
   and $h_\ell \le 1$, we obtain
   \[
       \mathcal{L}_\ell \le 2 h_\ell^{-2} \mathcal M_\ell + \mathcal A_\ell.
   \]
   Thus, using \eqref{eq:ap:2d} and the stability of
   the $\mathcal A$-orthogonal projector $I-\mathcal T_{\ell-1}$, we have
   \[
       \| (I-\mathcal{T}_{\ell-1}) u_\ell\|_{\mathcal{L}_\ell}^2 
       \le
       2 \| (I-\mathcal{T}_{\ell-1}) u_\ell\|_{ h_\ell^{-2} \mathcal{M}_\ell }^2 
         + \| (I-\mathcal{T}_{\ell-1}) u_\ell\|_{\mathcal{A}}^2 
     \le
       (2c + 1) \| u_\ell\|_{\mathcal{A}}^2
   \]
   for all $u_\ell \in \widetilde{S}_{p,\ell}(\Omega).$
\qed\end{proof}

We now immediately obtain the following convergence result.

\begin{theorem}\label{thm:inner-smoother-2d-modified}
    Consider the smoother~\eqref{eq:sm} with 
    \[
        \mathcal{L}_\ell = h_\ell^2 ( L_\ell\otimes L_\ell - C_\ell\otimes C_\ell ),
    \]
    where $L_\ell = h_\ell^{-2} M_\ell + C_\ell$ and
    $C_\ell = (I-T_\ell^I)^T A (I-T_\ell^I)$ as in Section~\ref{sec:1d},
    and $T_\ell^I: S_{p,\ell}(0,1) \to S_{p,\ell}^I(0,1)$
    is the $A$-orthogonal projector.
    Then, there is some damping parameter $\tau>0$ and some choice $\nu_0>0$, both independent of $\ell$
    and $p$, such that the two-grid method with $\nu>\nu_0$
    smoothing steps converges with rate $q=\nu_0/\nu$.
		
    With the same choice of $\tau$ and with $\nu > 4 \nu_0$, also the W-cycle multigrid
    method with $\nu/2$ pre- and $\nu/2$ post-smoothing steps converges in the energy norm with
    convergence rate $q=2\nu_0/\nu$.
\end{theorem}
\begin{proof}
    Lemma~\ref{lem:2d-inverse-mod} and Lemma~\ref{lem:ap:2d} show the
    assumptions \eqref{eq:sp:2} and \eqref{eq:ap:2}
    of Theorem~\ref{thrm:two:grid:gen1}, respectively.
    The statements then follow from Theorems~\ref{thrm:two:grid:gen1} and~\ref{thrm:two:grid:gen1a}.
\qed\end{proof}

To obtain a quasi-optimal method, we need a way of inverting
the matrix $\mathcal{L}_\ell$ efficiently.
Grouping the degrees of freedom as in~\eqref{eq:k:blocks}, we obtain
\begin{equation}\label{eq:def:k:hat:mod}
    \mathcal{L}_\ell
    = h_\ell^2 L_\ell \otimes L_\ell -  h_\ell^2 [P_\ell^T\otimes P_\ell^T][Q_\ell\otimes Q_\ell][P_\ell\otimes P_\ell]
\end{equation}
with
\begin{equation}\label{eq:def:q}
	P_\ell := \left(\begin{array}{c} I \\0 \end{array} \right)\qquad \text{and}\qquad Q_\ell:=A_{\Gamma\Gamma,\ell} - A_{I\Gamma,\ell}^T A_{II,\ell}^{-1}A_{I\Gamma,\ell}.
\end{equation}
Exploiting the Kronecker product structure,
$L_\ell \otimes L_\ell$ can be efficiently inverted
using the ideas from \cite{DeBoor:1979}.
The second term, $h_\ell^2 [P_\ell^T\otimes P_\ell^T][Q_\ell\otimes Q_\ell][P_\ell\otimes P_\ell]$,
is a low-rank correction that lives only near the four corners of the domain.
We can thus invert the smoother using the Sherman-Morrison-Woodbury formula and obtain
\[
	\mathcal{L}_\ell^{-1} = h_\ell^{-2} (I + ([L_\ell^{-1}P_\ell] \otimes [L_\ell^{-1}P_\ell])
								\mathcal{R}_\ell^{-1}
								(P_\ell^T\otimes P_\ell^T))(L_\ell^{-1} \otimes L_\ell^{-1}),
\]
where $P_\ell$ and $Q_\ell$ are defined as in~\eqref{eq:def:q},
\begin{align}
		\mathcal{R}_\ell  &:= Q_\ell^{-1} \otimes Q_\ell^{-1} - W_\ell^{-1} \otimes W_\ell^{-1},
          \label{def:mcR}\\
		W_\ell  &:= Q_\ell + h_\ell^{-2}( M_{\Gamma\Gamma,\ell} - M_{I\Gamma,\ell}^T M_{II,\ell}^{-1}M_{I\Gamma,\ell}),
          \label{eq:def:w}
\end{align}
and $M_{II,\ell}$, $M_{I\Gamma,\ell}$ and $M_{\Gamma\Gamma,\ell}$ are the blocks of the
mass matrix $M_\ell$ if the degrees of freedom are grouped as in~\eqref{eq:k:blocks}.


\section{Numerical realization and results}\label{sec:num}

\subsection{Numerical realization}\label{subsec:num:1}

As the proposed smoothing procedure appears rather complex, we give in Listing~\ref{listing:code}
pseudo code describing how the two-dimensional smoother based on the matrix
$\mathcal{L}_\ell$ as described in \eqref{eq:def:k:hat:mod}
can be implemented efficiently such that the overall multigrid method
achieves optimal complexity.
\begin{Listing}
\fbox{\parbox{.98\textwidth}{
\textbf{function} smoother()
\begin{itemize}
		\item \textbf{given matrices:}
        one-dimensional mass and stiffness matrices $M_\ell,A_\ell \in \mathbb{R}^{m_\ell\times m_\ell}$
		\item \textbf{input:} function value $\ul{u}_\ell^{(0,n)}\in \mathbb{R}^{m_\ell^2}$, corresponding residual $\ul{r}_\ell^{(0,n)}\in \mathbb{R}^{m_\ell^2}$
		\item Preparatory steps, only performed once:
				\begin{enumerate}
					\item Compute the Cholesky factorization of $A_{II,\ell}$.
					\item Determine $Q_\ell\in \mathbb{R}^{2p\times 2p}$ as defined in~\eqref{eq:def:q}.
					\item Compute the Cholesky factorization of $M_{II,\ell}$.
					\item Determine $W_\ell\in \mathbb{R}^{2p\times 2p}$ as defined in~\eqref{eq:def:w}.
					\item Determine $\mathcal{R}_\ell\in\mathbb{R}^{4p^2\times 4p^2}$ as defined in~\eqref{def:mcR}.
					\item Compute the Cholesky factorization of $\mathcal{R}_\ell$.
          \item Determine the sparse matrix $L_\ell\in\mathbb{R}^{m_\ell\times m_\ell} $ as defined in~\eqref{eq:def:khat}.
          \item Compute the Cholesky factorization of $L_\ell$.
				\end{enumerate}
		\item For each smoothing step do:
			\begin{enumerate}
					\item Determine $\ul{q}_\ell^{(0)} :=h_\ell^{-2}(L_\ell^{-1} \otimes L_\ell^{-1}) \ul{r}_\ell$.
					\item Determine $\ul{q}_\ell^{(1)} :=(P_\ell^T \otimes P_\ell^T) \ul{q}_\ell^{(0)}$.
					\item Determine $\ul{q}_\ell^{(2)} :=\mathcal{R}_\ell^{-1} \ul{q}_\ell^{(1)}$ using the Cholesky factorization of $\mathcal{R}_\ell$.
					\item Determine $\ul{q}_\ell :=([L_\ell^{-1}P_\ell] \otimes [L_\ell^{-1}P_\ell]) \ul{q}_\ell^{(2)}$
						using the Cholesky factorization of $L_\ell$.
					\item Set $\ul{p}_\ell =\ul{q}_\ell^{(0)}+\ul{q}_\ell$.
					\item Update the function $\ul{u}_\ell^{(0,n+1)} := \ul{u}_\ell^{(0,n)} + \tau \ul{p}_\ell$.
					\item Update the residual $\ul{r}_\ell^{(0,n+1)} := \ul{r}_\ell^{(0,n)} - \tau \mathcal{A}_\ell \ul{p}_\ell$.
			\end{enumerate}
		\item \textbf{output:} function value $\ul{u}_\ell^{(0,n+1)}$, corresponding residual $\ul{r}_\ell^{(0,n+1)}$
\end{itemize}
}}
\label{listing:code}
\caption{Pseudo-code implementation for the smoother in two dimensions}
\end{Listing}

The overall costs of the preparatory steps are $\mathcal{O}(m_\ell p^2+p^6)$ floating point operations,
where $m_\ell>p$ is the number of degrees of freedom in one dimension and $m_\ell^2$ is the overall number
of degrees of freedom.

Here, for the preparatory steps~1 and 3, $\mathcal{O}(m_\ell p^2)$ operations are required, as the dimension of the matrices
$A_{II,\ell}$ and $M_{II,\ell}$ is $m_\ell-2p = \mathcal{O}(m_\ell)$ and the bandwidth is $\mathcal{O}(p)$. 
For the preparatory steps~2 and 4, it is required to solve $\mathcal{O}(p)$ linear systems involving this
factorization, which requires again $\mathcal{O}(m_\ell p^2)$ operations. The preparatory steps~5 and 6 just live on
the 4 vertices of the square and require $\mathcal{O}(p^6)$ operations. The preparatory step~7 costs $\mathcal{O}(m_\ell p)$
operations just for adding $A_\ell$ and $L_\ell$. The Cholesky factorization to be performed
in the preparatory step~8 has -- as in the steps~1 and 3 -- a computational complexity of
$\mathcal{O}(m_\ell p^2)$ operations.

The steps~1 and 4 of the smoother itself require $\mathcal{O}(m_\ell^2 p)$ operations each, if the tensor product structure
$L_\ell^{-1} \otimes L_\ell^{-1} = (I \otimes L_\ell^{-1})(L_\ell^{-1} \otimes I)$
is used (see \cite{DeBoor:1979} for the algorithmic idea).
Step~2 of the smoother requires $\mathcal{O}(m_\ell p^2)$ operations. In step~3, only $\mathcal{O}(p^4)$ operations are required
as $\mathcal{R}_\ell$ is a dense $4p^2\times 4p^2$-matrix. The steps~5 and 6, which are only adding up, can be completed with
$\mathcal{O}(m_\ell^2)$ operations. Step~7 can be computed using the decomposition
\begin{equation}\label{eq:the:star:decomp} 
		\mathcal{A}_\ell=(I\otimes M_\ell)(K_\ell\otimes I) + (M_\ell\otimes I)(I\otimes K_\ell)+ (M_\ell\otimes I)(I\otimes M_\ell)
\end{equation}
with $\mathcal{O}(m_\ell^2p)$ operations.

In summary, the preparatory steps require $\mathcal{O}(m_\ell p^2+p^6)$ operations and the smoother itself requires
$\mathcal{O}(m_\ell^2 p)$ operations.
Furthermore, also the coarse-grid correction can be performed with $\mathcal{O}(m_\ell^2p)$ operations,
if the tensor product structure of the intergrid transfer matrices $\mathcal{I}_\ell^{\ell-1}$ and $\mathcal{I}_{\ell-1}^\ell$ is used.
The overall costs are therefore $\mathcal{O}(m_\ell^2 p+p^6)$ or,
assuming $p^5\le m_\ell^2$, $\mathcal{O}(m_\ell^2 p)$ operations.
The method can therefore be called asymptotically optimal,
since the multigrid solver requires $\mathcal{O}(m_\ell^2 p)$ operations in total,
which is the same effort as for multiplication with $\mathcal{A}_\ell$.

In the case of variable coefficients
or geometry transformations, $\mathcal{A}_\ell$ is no longer a sum of
tensor-product matrices. In this case, the computation of the residual of the original problem
requires $\mathcal{O}(m_\ell^2 p^2)$ operations, while the application of the multigrid method as a
preconditioner can still be performed with $\mathcal{O}(m_\ell^2 p)$ operations.

\subsection{Experimental results}\label{subsec:num:2}

As a numerical example, we solve the model problem from Section~\ref{subsec:2b}
with the right-hand side
\[
    f( \mathbf{x} ) = d \pi^2 \prod_{j=1}^d \sin(\pi (x_j + \tfrac12))
\]
and homogeneous Neumann boundary conditions
on the domain $\Omega=(0,1)^d$, $d=1,2$.
We perform a (tensor product) B-spline discretization using uniformly sized
knot spans and maximum-continuity splines for varying spline degrees $p$.
We refer to the coarse discretization with only a single interval as level $\ell=0$
and perform uniform, dyadic refinement to obtain the finer discretization levels $\ell$
with $2^{\ell d}$ elements (intervals or quadrilaterals).

We set up a V-cycle multigrid method according to the framework established in
Section~\ref{sec:mg} and using the proposed smoother \eqref{eq:def:khat}
for the one-dimensional domain and \eqref{eq:def:k:hat:mod} for the two
dimensional domain.
We always use one pre- and one post-smoothing step.
The damping parameter $\tau$ was chosen as $\tau=0.14$ for $d=1$ and $\tau=0.08$ for $d=2$.
In the one-dimensional case, we apply the damping factor
only to the mass component of the smoother. This slightly improves the iteration
numbers but is non-essential.

Due to the requirement that the number of intervals exceed $p$ so that the
space $S^I_{p,\ell}(0,1)$ is nonempty and the construction of the smoother is valid,
we cannot always use $\ell=0$ as the coarse grid for the multigrid method.

We perform two-grid iteration until the Euclidean norm
of the initial residual is reduced by a factor of $10^{-8}$.
The iteration numbers using varying spline degrees $p$ as well as varying
fine grid levels $\ell$ for the one-dimensional domain are given in Table~\ref{tab:iters-1d}.
Here we always used $\ell=5$ as the coarse grid since this is an
admissible choice for all tested $p$.
We observe that the iteration numbers are robust with respect to the grid size,
the number of levels, and the spline degree $p$.

In the two-dimensional experiments, for every degree $p$, we determine the
coarsest level on which $S^I_{p,\ell}(0,1)$ is nonempty and use the next coarser level
as the coarsest grid for the multigrid method.
The resulting iteration numbers are displayed in Table~\ref{tab:iters-2d}.
Here, the bottom-most iteration number in each column corresponds to a two-grid method,
the next higher one to a three-grid method, and so on.
Again, the iteration numbers remain uniformly bounded with respect to all discretization parameters,
although they are much higher than in the one-dimensional case. To mitigate this,
we can solve the problem with Conjugate Gradient (CG) iteration preconditioned with
one multigrid V-cycle. As the resulting iteration numbers in Table~\ref{tab:iters-2d-cg},
now only for $\ell=7$, show, this leads to a significant speedup while maintaining robustness.

\setlength{\tabcolsep}{4.8px}
\begin{table}[htbp]
    \centering
    \begin{tabular}[c]{r|cccccccccccccccccccc}
                & &&&&&&& $p$ \\
                &  1  &  2  &  3  &  4  &  5  &  6 &    7 &   8 &   9 &  10 &  11 &  12 &  13 &  14 &  15 \\ 
                \hline
            12  &  23 &  20 &  20 &  20 &  20 &  20 &  20 &  20 &  20 &  19 &  19 &  19 &  19 &  18 &  18 \\ 
  $\ell$ \; 11  &  23 &  20 &  20 &  20 &  20 &  20 &  20 &  20 &  19 &  19 &  19 &  19 &  18 &  19 &  18 \\ 
            10  &  23 &  20 &  20 &  20 &  20 &  20 &  20 &  19 &  19 &  19 &  19 &  18 &  17 &  17 &  17 \\ 
         \end{tabular}
    \caption{Multigrid iteration numbers in 1D for varying fine-grid levels and spline degrees.}
    \label{tab:iters-1d}
\end{table}

\begin{table}[htbp]
    \centering
    \begin{tabular}[c]{r|ccccccccccccccc}
               &      &       &        &       &       &       &       &  $p$  &      &       &       &       &       &       &       \\
               &  1   &   2   &   3    &  4    &  5    &  6    &  7    &  8   &   9   &   10  &   11  &   12  &   13  &   14  &   15  \\
          \hline
          7    &  86  &   88  &   99   &  102  &  99   &  100  &  99   &  98  &   97  &   96  &   94  &   95  &   93  &   92  &   92  \\
          6    &  84  &   89  &   101  &  104  &  100  &  101  &  100  &  97  &   97  &   96  &   94  &   94  &   94  &   93  &   92  \\
          5    &  83  &   92  &   103  &  103  &  100  &  100  &  101  &  97  &   97  &   96  &   94  &   93  &   94  &   91  &   91  \\
$\ell$ \; 4    &  66  &   95  &   104  &  105  &  102  &  100  &  99   &  96  &   96  &   95  &   94  &   92  &   92  &   91  &   91  \\
          3    &  45  &   97  &   105  &  107  &  103  &  101  &  101  &  -   &  -    &  -    &  -    &  -    &  -    &  -    &  -    \\
          2    &  32  &   97  &   114  &  -    &  -    &  -    &  -    &  -   &  -    &  -    &  -    &  -    &  -    &  -    &  -    \\
          1    &  32  &  -    &  -     &  -    &  -    &  -    &  -    &  -   &  -    &  -    &  -    &  -    &  -    &  -    &  -
    \end{tabular}
    \caption{Multigrid iteration numbers in 2D for varying fine-grid levels and spline degrees.}
    \label{tab:iters-2d}
\end{table}

\begin{table}[htbp]
    \centering
    \begin{tabular}{r|ccccccccccccccc}
        $p$  &  1   &   2   &   3    &  4    &  5    &  6    &  7    &  8   &   9   &   10  &   11  &   12  &   13  &   14  &   15  \\
        \hline
        $\ell=7$ &  21 & 21 & 23 & 23 & 23 & 22 & 23 & 22 & 22 & 22 & 21 & 21 & 21 & 21 & 21
    \end{tabular}
    \caption{Iteration numbers in 2D for CG preconditioned with V-cycle, varying $p$.}
    \label{tab:iters-2d-cg}
\end{table}


\section{Conclusions and outlook}\label{sec:conc}

We have analyzed in detail the convergence
properties of a geometric multigrid method for an isogeometric model
problem using B-splines.
In a first step, we discussed one-dimensional domains.
Based on the obtained insights, in particular on the importance of boundary
effects on the convergence rate, we then proposed a boundary-corrected
mass-Richardson smoother.
We have proved that this new smoother yields convergence rates which are
robust with respect to the spline degree, a result which is not attainable
with classical \cite{HofreitherZulehner:2014c}
or purely mass-based smoothers \cite{HofreitherZulehner:2014b}.

Exploiting the tensor product structure of spline spaces commonly used in IgA,
we extended the construction of the smoother to the two-dimensional setting
and again proved robust convergence for this case.
We have shown how the proposed smoother can be efficiently realized such
that the overall multigrid method has quasi-optimal complexity.
Although we have restricted ourselves to simple tensor product domains, our
technique is easily extended to problems with non-singular geometry mappings
as discussed in Remark~\ref{rem:1}.

The extension of this approach to three or more dimensions remains open.
In particular, the representation of the smoother as the Kronecker product of
one-dimensional smoothers plus some low-rank correction as in \eqref{eq:def:k:hat:mod}
encounters some difficulties in this case.
Therefore, the construction of robust and efficient smoothers for the three- and
higher-dimensional cases is left as future work.


\section*{Acknowledgments}

The work of the first and third author was supported by the
National Research Network ``Geometry + Simulation'' (NFN S117, 2012--2016),
funded by the Austrian Science Fund (FWF).
The work of the second author was funded by the Austrian Science Fund (FWF): J3362-N25.


\section*{Appendix}

Our aim in this section is to prove Theorem~\ref{thm:ap}.
First, however, we show a refined version of the main approximation
result from \cite{Takacs:Takacs:2015}.

\begin{theorem}\label{thrm:approx}
		Let $\ell \in \mathbb{N}_0$ and $p\in \mathbb{N}$. Then
		\[
			\|(I-\widetilde{\Pi}_\ell) u\|_{L^2(\Omega)} \le \sqrt{2} h_\ell |u|_{H^1(\Omega)} \qquad \forall u\in H^1(\Omega),
    \]
		where $\widetilde{\Pi}_\ell:H^1(\Omega)\rightarrow \widetilde{S}_{p,\ell}(\Omega)$
		is the $H^1_{\circ}(\Omega)$-orthogonal projector and
		\[
			(u,v)_{H^1_{\circ}(\Omega)}:=(\nabla u,\nabla v)_{L^2(\Omega)}
						+\frac{1}{|\Omega|}\left(\int_{\Omega}u(x)\dd x\right)\left(\int_{\Omega}v(x)\dd x\right).
		\]
\end{theorem}
\begin{proof}
    The proof is a slightly improved version of the proof of
    \cite[Theorem~1]{Takacs:Takacs:2015}. There, for any fixed $u\in H^1(\Omega)$, a function
    $u_{\ell}\in \widetilde{S}_{p,\ell}(\Omega)$ was constructed which satisfies the
    approximation error estimate
    \[
			\|u-u_{\ell}\|_{L^2(\Omega)} \le \sqrt{2} h_\ell |u|_{H^1(\Omega)}.
    \]
    (The space $\widetilde{S}_{p,\ell}(\Omega)$ coincides with the space $ \widetilde{S}_{p,h_\ell}(\Omega)$
    in~\cite{Takacs:Takacs:2015}.) Now, it remains to show that~$u_{\ell}$ coincides
    with~$\widetilde{\Pi}_\ell u$, i.e., that $u_{\ell}$ is the $H^1_{\circ}(\Omega)$-orthogonal
    projection of $u$ to the space $\widetilde{S}_{p,\ell}(\Omega)$.
    By definition, this means that we have to show that
    \begin{equation}\label{eq:thrm:approx:1}
    		( u-u_{\ell},\widetilde{u}_\ell)_{H^1_\circ(\Omega)}=0
            \qquad \forall \widetilde{u}_\ell \in \widetilde{S}_{p,\ell}(\Omega).
    \end{equation}
    Now, let $\widehat{S}_{p,\ell}(-1,1)$ be the space of periodic splines on
    $(-1,1)$ with degree $p$ and grid size $h_\ell$ (cf.~\cite[Definition~2]{Takacs:Takacs:2015}).
    Note that, by construction in~\cite[Theorem~1]{Takacs:Takacs:2015}, $u_{\ell}$ is
    the restriction of $w_{\ell}$ to $(0,1)$, where $w_{\ell}$ is the $H^1_{\circ}$-orthogonal projection of
    $w(x):=u(|x|)$ to $\widehat{S}_{p,\ell}(-1,1)$.
    Define $\widetilde{w} _\ell \in \widehat{S}_{p,\ell}(-1,1)$
    by $\widetilde{w}_\ell(x):=\widetilde{u}_\ell(|x|)$ and observe
		$
			 (w-w_{\ell},\widetilde{w}_\ell)_{H^1_\circ(-1,1)}=2(u-u_{\ell},\widetilde{u}_\ell)_{H^1_\circ(0,1)} .
		$		
		As $w_{\ell}$ is by construction the $H^1_{\circ}$-orthogonal projection of $w$, we have
		$(w-w_{\ell},\widetilde{w}_\ell)_{H^1_\circ(0,1)}=0$.
    This shows~\eqref{eq:thrm:approx:1} and therefore $u_{\ell} = \widetilde{\Pi}_\ell u$.
\qed\end{proof}

The error estimate for the $A$-orthogonal, i.e., the $H^1(\Omega)$-orthogonal, projector
$\widetilde{T}_{\ell}$ now follows immediately by a perturbation argument.

\begin{proof}[of Theorem~\ref{thm:ap}]
		Define, for $u , v\in H^1(\Omega)$, the scalar product
		\[
            (u,v)_B:=(u,v)_{A} - (u,v)_{H^1_\circ(\Omega)}
            = (u,v)_{L^2(\Omega)} - \frac{1}{|\Omega|}\left(\int_\Omega u \dd x\right)\left(\int_\Omega v \dd x\right)
		\]
		and observe that the Cauchy-Schwarz inequality implies
		$0\le (v,v)_B \le (v,v)_{L^2(\Omega)}$ for all $v\in H^1(\Omega)$.
        Let $u \in H^1(\Omega)$ arbitrary.
		Orthogonality implies
		\begin{align}
				&((I-\widetilde{\Pi}_\ell) u, w_{\ell} )_{H^1_{\circ}(\Omega)} = 0 \qquad &\forall  w_{\ell}\in \widetilde{S}_{p,\ell}(\Omega),\nonumber\\
				&((I-\widetilde{T}_\ell) u, w_{\ell} )_{A} = 0 \qquad &\forall  w_{\ell}\in \widetilde{S}_{p,\ell}(\Omega).   \label{corr:ap:2}
		\end{align}
        The first of these two equations implies
		\begin{align*}
				& ((I-\widetilde{\Pi}_\ell) u, w_{\ell} )_{A} = ((I-\widetilde{\Pi}_\ell) u, w_{\ell} )_B
				\qquad \forall  w_{\ell}\in \widetilde{S}_{p,\ell}(\Omega),
		\end{align*}
		and subtracting~\eqref{corr:ap:2} yields
		\begin{align*}
				& ((\widetilde{T}_\ell-\widetilde{\Pi}_\ell) u, w_{\ell} )_{A} = ((I-\widetilde{\Pi}_\ell) u, w_{\ell} )_B
				\qquad \forall  w_{\ell}\in \widetilde{S}_{p,\ell}(\Omega).
		\end{align*}
		With the choice $w_{\ell}:=(\widetilde{T}_\ell-\widetilde{\Pi}_\ell) u$
        as well as $\|\cdot\|_A=\|\cdot\|_{H^1(\Omega)}$ and the Cauchy-Schwarz inequality, we obtain
		\begin{align*}
				& \|(\widetilde{T}_\ell-\widetilde{\Pi}_\ell) u\|_{H^1(\Omega)}^2
							\le \|(I-\widetilde{\Pi}_\ell) u\|_B \| (\widetilde{T}_\ell-\widetilde{\Pi}_\ell) u\|_B,
		\end{align*}
		where $\|\cdot\|_B:=(\cdot,\cdot)_B^{1/2}$.
        Using $\|\cdot\|_B\le\|\cdot\|_{L^2(\Omega)}\le\|\cdot\|_{H^1(\Omega)}$,
        it follows that
		$
             \|(\widetilde{T}_\ell-\widetilde{\Pi}_\ell) u\|_{L^2(\Omega)} \le
              \|(I-\widetilde{\Pi}_\ell) u\|_{L^2(\Omega)}
		$.
		Thus, using the triangle inequality and Theorem~\ref{thrm:approx} we obtain
		\begin{align*}
            \|(I-\widetilde{T}_\ell) u\|_{L^2(\Omega)}
            &\le \|(I-\widetilde{\Pi}_\ell) u\|_{L^2(\Omega)} + \|(\widetilde{T}_\ell-\widetilde{\Pi}_\ell) u\|_{L^2(\Omega)}\\
            & \le 2 \|(I-\widetilde{\Pi}_\ell) u\|_{L^2(\Omega)} \le 2 \sqrt{2} h_{\ell} \|u\|_{H^1(\Omega)}.
            \qedhere
		\end{align*}
\end{proof}


\section*{Bibliography}

\bibliographystyle{amsplain}
\bibliography{literature}

\providecommand{\bysame}{\leavevmode\hbox to3em{\hrulefill}\thinspace}
\providecommand{\MR}{\relax\ifhmode\unskip\space\fi MR }
\providecommand{\MRhref}[2]{%
  \href{http://www.ams.org/mathscinet-getitem?mr=#1}{#2}
}
\providecommand{\href}[2]{#2}
\begin{thebibliography}{10}

\bibitem{Brenner:1996}
S.C. Brenner, \emph{{Multigrid methods for parameter dependent problems}},
  RAIRO, Mod\'elisation Math. Anal. Num\'er \textbf{30} (1996), 265 -- 297.

\bibitem{BuffaEtAl:2013}
A.~Buffa, H.~Harbrecht, A.~Kunoth, and G.~Sangalli, \emph{{BPX}-preconditioning
  for isogeometric analysis}, Computer Methods in Applied Mechanics and
  Engineering \textbf{265} (2013), 63--70.

\bibitem{Cottrell:2006}
J.A. Cottrell, A.~Reali, Y.~Bazilevs, and T.J.R. Hughes, \emph{Isogeometric
  analysis of structural vibrations}, Computer Methods in Applied Mechanics and
  Engineering \textbf{195} (2006), no.~41--43, 5257--5296, John H. Argyris
  Memorial Issue. Part {II}.

\bibitem{DeBoor:1972}
C.~de~Boor, \emph{On calculating with {B}-splines}, Journal of Approximation
  Theory \textbf{6} (1972), no.~1, 50--62.

\bibitem{DeBoor:1979}
\bysame, \emph{Efficient computer manipulation of tensor products}, ACM
  Transactions on Mathematical Software (TOMS) \textbf{5} (1979), no.~2,
  173--182.

\bibitem{DeBoor:Practical}
\bysame, \emph{A practical guide to splines (revised edition)}, Applied
  Mathematical Sciences, vol.~27, Springer, 2001.

\bibitem{Donatelli:2014b}
M.~Donatelli, C.~Garoni, C.~Manni, S.~Serra-Capizzano, and H.~Speleers,
  \emph{Robust and optimal multi-iterative techniques for {IgA} {Galerkin}
  linear systems}, Computer Methods in Applied Mechanics and Engineering
  \textbf{284} (2014), 230--264.

\bibitem{Donatelli:2014a}
\bysame, \emph{Symbol-based multigrid methods for {Galerkin} {B}-spline
  isogeometric analysis}, TW Reports (2014), no.~650.

\bibitem{Donatelli:2015}
\bysame, \emph{Robust and optimal multi-iterative techniques for {IgA}
  collocation linear systems}, Computer Methods in Applied Mechanics and
  Engineering \textbf{284} (2015), 1120--1146.

\bibitem{GahalautEtAl:2013}
K.~P.~S. Gahalaut, J.~K. Kraus, and S.~K. Tomar, \emph{Multigrid methods for
  isogeometric discretization}, Computer Methods in Applied Mechanics and
  Engineering \textbf{253} (2013), 413--425.

\bibitem{Hackbusch:1985}
W.~Hackbusch, \emph{{Multi-Grid Methods and Applications}}, Springer, Berlin,
  1985.

\bibitem{HofreitherZulehner:2014b}
C.~Hofreither and W.~Zulehner, \emph{Mass smoothers in geometric multigrid for
  isogeometric analysis}, NuMa Report 2014-11, Institute of Computational
  Mathematics, Johannes Kepler University Linz, Austria, 2014, Accepted for
  publication in the proceedings of the 8th International Conference Curves and
  Surfaces.

\bibitem{HofreitherZulehner:2014}
\bysame, \emph{On full multigrid schemes for isogeometric analysis}, NuMa
  Report 2014-03, Institute of Computational Mathematics, Johannes Kepler
  University Linz, Austria, 2014, Accepted for publication in the proceedings
  of the 22nd International Conference on Domain Decomposition Methods.

\bibitem{HofreitherZulehner:2014c}
\bysame, \emph{Spectral analysis of geometric multigrid methods for
  isogeometric analysis}, Numerical Methods and Applications (I.~Dimov,
  S.~Fidanova, and I.~Lirkov, eds.), Lecture Notes in Computer Science, vol.
  8962, Springer International Publishing, 2015, pp.~123--129.

\bibitem{Hughes:2005}
T.~J.~R. Hughes, J.~A. Cottrell, and Y.~Bazilevs, \emph{Isogeometric analysis:
  {CAD}, finite elements, {NURBS}, exact geometry and mesh refinement},
  Computer Methods in Applied Mechanics and Engineering \textbf{194} (2005),
  no.~39-41, 4135--4195.

\bibitem{Takacs:Takacs:2015}
S.~Takacs and T.~Takacs, \emph{Approximation error estimates and inverse
  inequalities for {B}-splines of maximum smoothness}, Preprint
  arXiv:1502.03733, Institute of Computational Mathematics, Johannes Kepler
  University Linz, Austria, 2015, Submitted.

\end{thebibliography}



\end{document}